\documentclass{amsart}

\usepackage[nobysame]{amsrefs}

\AtBeginDocument{\def\MR#1{}}

\usepackage[UKenglish]{babel}
\usepackage[UKenglish]{datetime}

\usepackage{ae, aecompl}
\usepackage{amscd}
\usepackage{amsfonts}
\usepackage{amsmath}
\usepackage{amssymb}
\usepackage{amsthm}
\usepackage{cases}
\usepackage{chngcntr}
\usepackage{cite}
\usepackage{color}
\usepackage{graphicx}
\usepackage{latexsym}
\usepackage{bbm} 
\usepackage{bbold}
\usepackage{microtype}
\usepackage{qsymbols}
\usepackage[Symbolsmallscale]{upgreek} 
\usepackage[active]{srcltx}
\usepackage{tikz}
\usepackage{url}
\usepackage{verbatim} 

\usepackage{enumitem}
\usepackage{cleveref}

\usepackage[charter]{mathdesign}

\newcommand{\bbfont}{\mathbbm}

\newcommand{\TT}{{\bbfont T}}

\newcommand{\ZZ}{{\bbfont Z}}

\newcommand{\upC}{{\mathrm{C}}}

\newcommand{\braces}[1]{{\{ #1\}}}

\newcommand{\lrbraces}[1]{{\left\{ #1\right\}}}

\newcommand{\lrpars}[1]{{\left( #1\right)}}

\newcommand{\biggbraces}[1]{{\bigg\{ #1\bigg\}}}

\newcommand{\set}[1]{\braces{\,#1\,}}
\newcommand{\lrset}[1]{\lrbraces{\,#1\,}}

\newcommand{\biggset}[1]{\biggbraces{\,#1\,}}

\newcommand{\cont}{{\upC}}

\newcommand{\Calgebra}{\ensuremath{{\upC}^\ast}-algebra}
\newcommand{\Calgebras}{\ensuremath{{\upC}^\ast}-algebras}

\newcommand{\onefunction}{{\mathbf 1}}
\newcommand{\zerofunction}{{\mathbf 0}}

\newcommand{\Ker}{\operatorname{Ker}}

\theoremstyle{plain}

\newtheorem{theorem}{Theorem}[section]
\newtheorem{proposition}[theorem]{Proposition}

\newtheorem{corollary}[theorem]{Corollary}

\newtheorem*{theorem*}{Theorem}
\newtheorem*{proposition*}{Proposition}
\newtheorem*{lemma*}{Lemma}
\newtheorem*{corollary*}{Corollary}
\newtheorem*{conjecture*}{Conjecture}

\theoremstyle{definition}

\newtheorem{remark}[theorem]{Remark}

\newtheorem*{definition*}{Definition}
\newtheorem*{example*}{Example}
\newtheorem*{remark*}{Remark}
\newtheorem*{assumption*}{Assumption}

\setlist[enumerate,1]{label=\textup{(\arabic*)},ref=\arabic*}
\setlist[enumerate,2]{label=\textup{(\alph*)},ref=\arabic{enumi}.\alph*}
\setlist[enumerate,3]{label=\textup{(\roman*)},ref=\arabic{enumi}.\alph{enumii}.\roman*}
\setlist[enumerate,4]{label=\textup{(\Alph*)},ref=\arabic{enumi}.\alph{enumii}.\roman{enumiii}.\Alph*}

\crefname{theorem}{Theorem}{Theorems}
\crefname{proposition}{Proposition}{Propositions}
\crefname{lemma}{Lemma}{Lemmas}
\crefname{corollary}{Corollary}{Corollaries}
\crefname{conjecture}{Conjecture}{Conjectures}
\crefname{definition}{Definition}{Definitions}
\crefname{example}{Example}{Examples}
\crefname{remark}{Remark}{Remarks}
\crefname{assumption}{Assumption}{Assumptions}

\crefname{equation}{equation}{equations}
\crefname{enumi}{part}{parts}
\crefname{enumii}{part}{parts}
\crefname{enumiii}{part}{parts}
\crefname{enumiv}{part}{parts}

\newcommand{\enclosepart}[1]{(#1)}
\newcommand{\partref}[1]{\enclosepart{\ref{#1}}}

\numberwithin{equation}{section}

\allowdisplaybreaks 

\newcommand{\ellone}{\ell^1(\Sigma)}
\newcommand{\cstar}{{\upC}^\ast(\Sigma)}

\newcommand{\perpoints}{\mathrm{Per}(\sigma)}
\newcommand{\aperpoints}{\mathrm{Aper}(\sigma)}

\newcommand{\elloneclosure}[1]{{\overline{#1}}{\vphantom{#1}}^{\ell^1}\!\!}
\newcommand{\cstarclosure}[1]{\overline{#1}{\vphantom{#1}}^{{\mathrm C}^\ast}\!\!}

\newcommand{\bounded}{\mathrm{B}}

\newcommand{\coeffalg}{\cont(X)}

\begin{document}


\title [Closure of ideals]{The closure of ideals of $\boldsymbol{\ell^1(\Sigma)}$ in its enveloping $\boldsymbol{\mathrm{C}^\ast}$-algebra}

\author{Marcel de Jeu}
\address{Marcel de Jeu, Mathematical Institute, Leiden University, P.O.\ Box 9512, 2300 RA Leiden, the Netherlands}
\email{mdejeu@math.leidenuniv.nl}

\author{Jun Tomiyama}
\address{Jun Tomiyama, Department of Mathematics, Tokyo Metropolitan University, Minami-Osawa, Hachioji City, Japan}
\email{juntomi@med.email.ne.jp}







\begin{abstract}
If~$X$ is a compact Hausdorff space and~$\sigma$ is a homeomorphism of~$X$, then an involutive Banach algebra $\ell^1(\Sigma)$ of crossed product type is naturally associated with the topological dynamical system $\Sigma=(X,\sigma)$. We initiate the study of the relation between two-sided ideals of $\ell^1(\Sigma)$ and  ${\mathrm C}^\ast(\Sigma)$, the enveloping $\mathrm{C}^\ast$-algebra ${\mathrm C}(X)\rtimes_\sigma
\mathbb Z$ of $\ell^1(\Sigma)$.  Among others, we prove that the closure of a proper two-sided ideal of $\ell^1(\Sigma)$ in  ${\mathrm C}^\ast(\Sigma)$ is again a proper two-sided ideal of  ${\mathrm C}^\ast(\Sigma)$.
\end{abstract}

\maketitle


\section{Introduction and overview}\label{sec:introduction}
If~$X$ is a compact Hausdorff space and $\sigma:X\to X$ is a homeomorphism of~$X$, then an involutive Banach algebra $\ell^1(\Sigma)$ of crossed product type can be associated with this dynamical system $\Sigma=(X,\sigma)$; we shall recall the definition in Section~\ref{sec:two_algebras}. Its ${\upC}^\ast$-enveloping algebra, denoted by $\cstar$, is the crossed product $\upC^\ast$-algebra ${\mathrm C}(X)\rtimes_\sigma
\mathbb Z$. Whereas $\cstar$ is well-studied, the investigation of $\ellone$ itself is of a more recent nature; this has been taken up in \cite{de_jeu_svensson_tomiyama:2012}, \cite{de_jeu_tomiyama:2012}, \cite{de_jeu_tomiyama:2013b}, and \cite{de_jeu_tomiyama:2016}. The algebra $\ellone$ is more complicated than $\cstar$. For example, it can occur that $\ellone$ has a non-selfadjoint closed two-sided ideal (see \cite[Theorem~4.4]{de_jeu_svensson_tomiyama:2012}), whereas this is, of course, never the case for $\cstar$.

The study of $\ellone$ so far has proceeded without using what is known about $\cstar$. Still, it has turned out that some analogous properties of $\ellone$ and $\cstar$ are equivalent. For example, these algebras are either both simple (i.e.\ have only trivial closed two-sided ideals), or both non-simple. The proof of this fact proceeds via the properties of~$\Sigma$: for each of these algebras, this simplicity can be shown to~$X$ being an infinite set and $\Sigma$ being minimal, i.e.\ to~$X$ being an infinite set that has only trivial invariant closed subsets. Hence the simplicity of one algebra also implies the simplicity of the other. It is desirable to have some basic results, formulated directly in terms of the algebras and not using the properties of the dynamical system, that can help transfer a property of the ideal structure of one algebra to an analogous property of the ideal structure of the other algebra. As the example of non-selfadjoint closed two-sided ideals makes clear, one cannot expect to be able to do this in all cases, but one would hope that something in this direction is still possible.

The present paper contains the first steps in this direction. We show that the closure in $\cstar$ of a proper not necessarily closed two-sided ideal of $\ellone$ is still a proper two-sided ideal of $\cstar$; see \cref{res:proper_closure}. We also investigate what the necessary and sufficient condition is so that all closed two-sided ideals of $\ellone$ can be reconstructed from their closure in $\cstar$ by taking the intersection with $\ellone$ again. The latter is possible if and only if the~$\ZZ$-action on~$X$ is free, i.e.\ if and only if there are no periodic points of~$\sigma$ in~$X$; see \cref{res:intersection_closure}.

As a (now immediate) example of how such a relation between two-sided ideals of $\ellone$ and $\cstar$ can be exploited, we re-establish the fact that the minimality of~$\Sigma$ implies the simplicity of $\ellone$, based on the knowledge that this is true for $\cstar$; see \cref{res:simplicity} for a more elaborate statement.

This paper is organized as follows.

In Section~\ref{sec:two_algebras}, we introduce $\ellone$ and its enveloping \Calgebra\ $\cstar$.

In Section~\ref{sec:results}, we establish our main results. The key idea in this section is that there are some proper two-sided ideals of $\ellone$ for which it is immediately obvious that their closure in $\cstar$ is still a proper two-sided ideal of $\cstar$, and that they can be retrieved as the intersection of their closure in $\cstar$ with $\ellone$. As is the case for all unital involutive Banach algebras that are a subalgebra of their enveloping \Calgebra, all kernels of non-zero involutive representations of $\ellone$ on Hilbert spaces have these two properties. The main technical result in this section is then to show that all primitive ideals (in the purely algebraic sense to be recalled in Section~\ref{sec:results}) of $\ellone$ are such kernels.

\section{Two algebras}\label{sec:two_algebras}

In this section, we introduce the two algebras playing a role in this paper, and we recall a few standard facts about their relation.  We also introduce notation for the two closure operations figuring in this context, and state our convention concerning ideals of an algebra.

Throughout this paper,~$X$ is a non-empty compact Hausdorff space and $\sigma:X\to X$ is a homeomorphism. Hence~$\ZZ$ acts on~$X$, and we write~$\Sigma$ for short for the topological dynamical system $(X,\sigma)$. We let $\aperpoints$ and $\perpoints$ denote the aperiodic and the periodic points of~$\sigma$, respectively.  A subset~$S$ of~$X$ is
\emph{invariant} if it is invariant under the~$\ZZ$-action, i.e.\  if $\sigma(S)=S$.
The involutive algebra of continuous (complex-valued) functions on~$X$ is denoted by $\coeffalg$, and we write $\alpha$ for the involutive automorphism of $\coeffalg$ induced by~$\sigma$, defined by $\alpha(f) = f \circ \sigma^{-1}$ for $f \in\coeffalg$. Via $n \mapsto \alpha^n$,~$\ZZ$ acts on $\coeffalg$.

With $\Vert\cdot\Vert$ denoting the supremum norm on $\coeffalg$, we let
\[
\ellone = \ell^1(\ZZ,\coeffalg)=\biggset{a: \ZZ \to\coeffalg : \Vert a \Vert := \sum_{n\in\ZZ} \Vert a(n)\Vert < \infty}.
\]
We supply the Banach space $\ellone$ with the usual twisted convolution as multiplication, defined by
\[
(aa^\prime) (n) = \sum_{k\in\ZZ}  a(k) \cdot \alpha^k (a^\prime(n-k))
\]
for $n\in\ZZ$ and $a, a^\prime \in \ellone$, and define an involution on $\ellone$ by
\[
a^* (n) = \overline{\alpha^n (a(-n))}
\]
for $n\in\ZZ$ and $a\in\ellone$. Thus $\ellone$ becomes a unital Banach $^\ast$-algebra with isometric involution. If~$X$ consists of one point, then $\ellone$ is the group algebra $\ell^1(\ZZ)$ of the integers.

There is a convenient way to work with $\ellone$, which we shall now explain. For $n,m \in \mathbb{Z}$, let
\[
 \chi_{\set{n}} (m) =
 \begin{cases}
   \onefunction &\text{if }m =n;\\
   \zerofunction &\text{if }m \neq n,
 \end{cases}
\]
where the constants denote the corresponding constant functions on~$X$. Then $\chi_{\set{0}}$ is the identity element of $\ellone$. Let $\delta = \chi_{\set{1}}$; then one sees easily that $\chi_{\set{-1}}=\delta^{-1}=\delta^*$. If we put $\delta^0=\chi_{\set{0}}$, then one computes that $\delta^n = \chi_{\set{n}}$ for all $n \in \mathbb{Z}$. We may view $\coeffalg$ as a closed abelian $^\ast$-subalgebra of $\ellone$, namely as $\set{a_0 \delta^0 \, : \, a_0 \in \coeffalg}$. If $a \in \ellone$, and if we write $f_n$ for  $a(n)$ as a more intuitive notation, then $a= \sum_{n\in\ZZ}f_n\delta^n$ and $\Vert a \Vert=\sum_{n\in\ZZ} \Vert f_n\Vert<\infty$.
In the rest of this paper, we shall constantly use this series representation $a= \sum_{n\in\ZZ}f_n\delta^n$ of an arbitrary element $a\in\ellone$, with uniquely determined $f_n\in\coeffalg$ for $n\in\ZZ$. Thus, all in all, $\ellone$ is generated, as a unital Banach algebra, by an isometrically isomorphic copy of $\coeffalg$ and the elements $\delta$ and $\delta^{-1}$, subject to the relation $\delta f \delta^{-1}=\alpha (f)= f\circ\sigma^{-1}$ for $f \in\coeffalg$. The isometric involution is determined by $f^*=\overline f$ for $f\in\coeffalg$ and by $\delta^*=\delta^{-1}$.

Since $\ellone$ is a unital Banach algebra with an isometric involution, it has an enveloping \Calgebra\ as constructed in \cite[Section~2.7]{dixmier_C-STAR-ALGEBRAS_ENGLISH_NORTH_HOLLAND_EDITION:1977}. We denote this enveloping ${\upC}^\ast$-algebra of $\ellone$ by $\cstar$. As in the general construction of crossed products of $\upC^\ast$-algebras (see \cite{williams_CROSSED_PRODUCTS_OF_C-STAR-ALGEBRAS:2007}), the enveloping ${\upC}^\ast$-seminorm on $\ellone$ is actually a norm; in this particular case, a somewhat shorter argument can also be used to see this (see \cite[p.~51]{de_jeu_tomiyama:2012}). Hence $\ellone$ can be viewed as a dense subalgebra of $\cstar$, and the inclusion of $\ellone$ into $\cstar$ is continuous (even contractive).

We infer from \cite[Proposition~1.3.7]{dixmier_C-STAR-ALGEBRAS_ENGLISH_NORTH_HOLLAND_EDITION:1977}  that a not necessarily unital involutive representation of $\ellone$ or $\cstar$ as bounded operators on a Hilbert space is always continuous (even contractive). As another standard fact, let us note that \cite[Proposition~2.7.4]{dixmier_C-STAR-ALGEBRAS_ENGLISH_NORTH_HOLLAND_EDITION:1977} shows that every not necessarily unital involutive representation of $\ellone$ extends uniquely to an involutive representation of $\cstar$, and that a bijection between the collections of involutive representations of the two algebras is thus obtained.

If $S\subset\ellone$, then $\elloneclosure{S}$ denotes its closure in $\ellone$. If $S\subset\cstar$, then $\cstarclosure{S}$ denotes its closure in $\cstar$.

Although we have written `two-sided ideal' in the abstract and in Section~\ref{sec:introduction} to avoid any possible misunderstanding, in the sequel of this paper, an \emph{ideal} of an algebra is always a two-sided ideal, unless otherwise stated. It need not (if applicable) be closed.

\section{Results}\label{sec:results}

In this section, we take up the study of the relation between the ideals of $\ellone$ and $\cstar$. We shall show that the closure $\cstarclosure{I}$ of a proper ideal of $\ellone$ is still a proper ideal of $\cstar$; see \cref{res:proper_closure}. Furthermore, we shall investigate when closed ideals of $\ellone$ can be recovered from their closure in $\cstar$ in the most obvious fashion; see \cref{res:intersection_of_primitive_ok,res:intersection_closure}. The key technical result of this section is the fact that every primitive ideal of $\ellone$ is the kernel of a non-zero involutive representation of $\ellone$; see \cref{res:primitive_and_involutive}.

The proof of \cref{res:proper_closure} can be reduced to a particular case by the following purely algebraic argument.

We recall that, if $A$ is an algebra, an \emph{algebraically irreducible representation of $A$} is a non-zero homomorphism into the linear operators on a vector space $E$ over the pertinent field such that $E$ has only trivial invariant subspaces. A \emph{primitive ideal of $A$} is the kernel of an algebraically irreducible representation of $A$. It was already noted by Jacobson (see \cite{jacobson:1945}) that a maximal ideal~$M$ of a unital algebra $A$ is a primitive ideal. Since it is this fact that allows us to make a reduction that is instrumental for the proof of \cref{res:proper_closure}, we recall the short argument, which is as follows. Since $A$ is unital, the proper left ideal~$M$ is contained in a maximal left ideal~$I$. The representation of $A$ on $A/I$ is then algebraically irreducible, so that $\Ker(\pi)$ is a primitive ideal. Since~$M$ is also a right ideal of $A$, one has $M\subset\Ker(\pi)$. By the maximality of~$M$, we conclude that $M=\Ker(\pi)$. Hence~$M$ is a primitive ideal, as desired. Since in a unital algebra every proper ideal is contained in a maximal ideal, we see that in a unital algebra every proper ideal is contained in a primitive ideal.

As a consequence of these well-known results, if we want to prove that $\cstarclosure{I}$\, is a proper subset of $\cstar$ for every proper ideal~$I$ of $\ellone$, then it is sufficient to do this when~$I$ is a primitive ideal of $\ellone$. In order to fully exploit this reduction, it is clearly important to have more information about the primitive ideals of $\ellone$, and we shall now set out to collect the relevant facts from previous work. Before doing so, however, let us note that, in the literature on involutive Banach algebras, a primitive ideal is often defined as the kernel of a topologically irreducible (i.e.\ having only trivial closed invariant subspaces) non-zero involutive representation of the pertinent algebra on a Hilbert space. This was also the notion employed by the authors in \cite{de_jeu_tomiyama:2013b}, but in \cite{de_jeu_tomiyama:2016} the purely algebraic notion as in the present paper was used. For \Calgebras, there is no difference (this follows from the combination of \cite[Corollary~2.8.4]{dixmier_C-STAR-ALGEBRAS_ENGLISH_NORTH_HOLLAND_EDITION:1977} and \cite[Corollary~2.9.6.(i)]{dixmier_C-STAR-ALGEBRAS_ENGLISH_NORTH_HOLLAND_EDITION:1977}), but otherwise a little care is in order when using results (including those of the authors) for primitive ideals of involutive Banach algebras as found in the literature.

We start by describing a family of finite dimensional algebraically irreducible representations (hence also of primitive ideals) that are associated with periodic points of~$\sigma$. When combined with the standard relation between the involutive representations of $\ellone$ and $\cstar$ on Hilbert spaces, the existence of these representations also follows from general considerations for crossed products of $\upC^\ast$-algebras (see \cite{williams_CROSSED_PRODUCTS_OF_C-STAR-ALGEBRAS:2007}), but the direct definition below suffices for our needs. More can be said about these representations of $\ellone$ than we shall include here, and we refer to \cite[Section~3.2]{de_jeu_tomiyama:2016} for further information.

For $x\in\perpoints$ and $\lambda\in\TT$, we define a representation $\pi_{x,\lambda}$ of $\ellone$ as follows. Let~$p$ be the period of~$X$. Let $H_{x,\lambda}$ be a Hilbert space with orthonormal basis $\set{e_0,\ldots,e_{p-1}}$ and bounded operators $\bounded(H_{x,\lambda})$. We let $T_\lambda\in\bounded(H_{x,\lambda})$ be the bounded linear operator on $H_{x,\lambda}$ that is represented with respect to this basis by the matrix
\[
\left(
\begin{array}{ccccc}
0 & 0 & \ldots & 0 & \lambda \\
1 & 0 & \ldots & 0 & 0 \\
0 & 1 & \ldots & 0&0\\
\vdots & \vdots & \ddots &\vdots &\vdots \\
0 & 0 & \ldots & 1&0
\end{array}\right).
\]
For $f\in\coeffalg$, we let $\rho_{x }(f)$ be the bounded linear operator on $H_{x,\lambda}$ that is represented with respect to this basis by the matrix
\[
\left(
\begin{array}{cccc}
f({x }) & 0 & \ldots & 0 \\
0 & f (\sigma {x }) & \ldots & 0 \\
\vdots & \vdots & \ddots & \vdots \\
0 & 0 & \ldots & f(\sigma^{p-1} {x })
\end{array} \right).
\]
It is easily checked (we refer to \cite[Lemma~3.1.2]{de_jeu_tomiyama:2016} for details), that there exists a unique involutive representation $\pi_{x ,\lambda}:\ellone\to\bounded(H_{x ,
\lambda})$ such that $\pi_{x ,\lambda}\upharpoonright_{\coeffalg}=\rho_{x }$ and $\pi(\delta)=T_\lambda$.

In \cite[Theorem~3.5]{de_jeu_tomiyama:2016}, the finite dimensional algebraically irreducible representations of $\ellone$ are classified. Amongst others, this classifications shows, somewhat surprisingly, that the Hilbert space context is automatic for finite dimensional algebraically irreducible representations of $\ellone$. This is the first part of the following result. The description of the intersection of primitive ideals in \cref{eq:intersection_of_kernels_of_finite_dimensional_representation_for_periodic_point}   follows from an explicit description of $\Ker(\pi_{x,\lambda})$ and the injectivity of the Fourier transform on $\ell^1(\ZZ)$; see \cite[Proposition~2.10]{de_jeu_tomiyama:2013b}.

\begin{proposition}\label{res:involutive_periodic_point}
If~$\pi$ is a finite dimensional algebraically irreducible representation of $\ellone$, then there exist $x\in\perpoints$ and $\lambda\in\TT$ such that~$\pi$ and $\pi_{x,\lambda} $ are algebraically equivalent. In particular, $\Ker(\pi)=\Ker(\pi_{x,\lambda})$. Hence every primitive ideal of $\ellone$ that arises as the kernel of a finite dimensional algebraically irreducible representation of $\ellone$ is also the kernel of a topologically irreducible non-zero involutive representation of $\ellone$ on a Hilbert space, and it is a selfadjoint ideal.

Furthermore,
\begin{equation}\label{eq:intersection_of_kernels_of_finite_dimensional_representation_for_periodic_point}
\bigcap_{\lambda\in\TT}\Ker(\pi_{x,\lambda})=\biggset{\sum_{n\in\ZZ}f_n\delta^n\in\ellone : f_n(\sigma^k x)=0 \textup{ for all } n,k\in\ZZ}.
\end{equation}
\end{proposition}

\begin{remark}
Given the algebraic irreducibility, the topological irreducibility of the finite dimensional representations in \cref{res:involutive_periodic_point} is, of course, immediate. We have nevertheless included it, since an analogous statement is also true when this is less obvious (see part~\partref{aperiodic_point_kernel_of_involutive_representation} of  \cref{res:infinite_dimensional}). Although it is not relevant for the proofs of \cref{res:proper_closure},  \cref{res:intersection_of_primitive_ok}, or \cref{res:intersection_closure}, the presence of this topological irreducibility seems too remarkable not to include it; see also part~\partref{kernel_of_topologically_irreducible_involutive_representation} of \cref{res:primitive_and_involutive}.
\end{remark}

Now that all primitive ideals corresponding to finite dimensional algebraically irreducible representations have been described, and have been related to involutive representations, we turn to the infinite dimensional case. In \cite{de_jeu_tomiyama:2016} it was repeatedly used that the representation space of an algebraically irreducible representation of $\ellone$ can always be normed in such a way that the algebra acts as continuous operators and that the representation is a bounded map. This is an immediate consequence (see e.g.\ \cite[proof of Lemma~25.2]{bonsall_duncan_COMPLETE_NORMED_ALGEBRAS:1973}) of the fact that a maximal left ideal in a unital Banach algebra is closed. Combining this normability (which also shows that primitive ideals of $\ellone$ are closed) with \cite[Propositions~2.6,~4.2, and~4.17]{de_jeu_tomiyama:2016} yields the following.

\begin{proposition}\label{res:kernel_of_infinite_dimensional_algebraically_irrreducible_representation}
If~$\pi$ is an infinite dimensional algebraically irreducible representation of $\ellone$, then there exists an infinite invariant subset~$S$ of~$X$ such that
\begin{equation}\label{eq:kernel_of_infinite_dimensional_algebraically_irreducible_representation}
\Ker(\pi)=\biggset{\sum_{n\in\ZZ} f_n\delta^n\in\ellone: f_n\upharpoonright_{S}=0\textup{ for all }n\in\ZZ}.
\end{equation}
Hence $\Ker(\pi)$ is a self-adjoint ideal.

If~$X$ is metrizable, then there exists $x\in\aperpoints$ such that the description of~$I$ in \cref{eq:kernel_of_infinite_dimensional_algebraically_irreducible_representation} holds with $S=\set{\sigma^n x : n\in\ZZ}$.
\end{proposition}

In order to relate $\Ker(\pi)$ for an infinite dimensional algebraically irreducible representation $\pi$ to an involutive representation, we shall use \cref{eq:kernel_of_infinite_dimensional_algebraically_irreducible_representation}, the representations $\pi_{x,\lambda}$ as described before for periodic points $x$ and $\lambda\in\TT$, and a family of infinite dimensional representations of $\ellone$ that are associated with aperiodic points and that we shall now introduce. Here, again, more can be said than we shall include in the present paper, and we refer to \cite[Section~3.3]{de_jeu_tomiyama:2016} for further information.

Fix $x\in\aperpoints$, and let $1\leq p<\infty$. We let $\bounded(\ell^p(\ZZ))$ denote the bounded linear operators on $\ell^p(\ZZ)$, and, for $k \in \ZZ$, we let $e_k$ denote the element of $\ell^p(\ZZ)$ with 1 in the $k$th coordinate and zero elsewhere. Let $S\in\bounded(\ell^p(\ZZ))$ be the right shift, determined by $Se_k=e_{k+1}$ for $k\in\ZZ$. For $f\in\coeffalg$, let $\pi^p_x (f)\in\bounded(\ell^p(\ZZ))$ be determined by $ \pi^p_x (f)e_k= f(\sigma^k x  )e_k$ for $k\in\ZZ$. It is then easily seen (see \cite[Lemma~3.1]{de_jeu_tomiyama:2016} for details) that there exists a unique unital continuous  representation $\pi^p_x :\ellone\to\bounded(\ell^p(\ZZ))$ such that
\begin{equation}\label{eq:aperiodic_point_representation}
\pi^p_x \left(\sum_{n\in\ZZ} f_n\delta^n\right)=\sum_{n\in\ZZ}\pi^p_x (f)S^n\quad
\end{equation}
for all $\sum_{n\in\ZZ} f_n\delta^n\in\ellone$. Furthermore, if $p=2$, then $\pi_x^2$ is an involutive representation of $\ellone$ on the Hilbert space $l^2(\ZZ)$.

We collect a few relevant facts about these infinite dimensional representations, the first three of which are taken from \cite[Theorem~3.16 and Lemma~4.9]{de_jeu_tomiyama:2016}. Part~\partref{aperiodic_point_topologically_irreducible} is not too hard to establish, and part~\partref{aperiodic_point_kernel_independent} is rather obvious, but part~\partref{aperiodic_point_algebraically_irreducible} is considerably more intricate. Part~\partref{aperiodic_point_kernel_of_involutive_representation} is immediate from the parts~\partref{aperiodic_point_algebraically_irreducible},~\partref{aperiodic_point_topologically_irreducible}, and~\partref{aperiodic_point_kernel_independent}.

\begin{proposition}\label{res:infinite_dimensional} Let $x\in\aperpoints$ and let $1\leq p<\infty$. Then:
\begin{enumerate}
\item\label{aperiodic_point_algebraically_irreducible} The representation $\pi_x^1$ of $\ellone$ on $\ell^1(\ZZ)$ is algebraically irreducible;
\item\label{aperiodic_point_topologically_irreducible} For $1<p<\infty$, the representation $\pi_x^p$ of $\ellone$ on $\ell^p(\ZZ)$ is topologically  irreducible, but not algebraically irreducible;
\item\label{aperiodic_point_kernel_independent} The kernel of $\pi_x^p$ is a selfadjoint ideal of $\ellone$ that does not depend on~$p$. In fact,
\begin{equation}\label{eq:aperiodic_point_kernel_description}
\Ker(\pi_x^p)=\lrset{\sum_{n\in\ZZ}f_n\delta^n\in\ellone : f_n(\sigma^k x)=0 \textup{ for all } n,k\in\ZZ};
\end{equation}
\item\label{aperiodic_point_kernel_of_involutive_representation} The primitive ideal $\Ker(\pi_x^1)$ of $\ellone$ is also the kernel of the topologically irreducible involutive representation $\pi_{x}^2$ of $\ellone$ on the Hilbert space $\ell^2(\ZZ)$.
\end{enumerate}
\end{proposition}

\begin{remark}
Before proceeding, let us note that, if~$X$ is metrizable, the combination of \cref{res:involutive_periodic_point,res:kernel_of_infinite_dimensional_algebraically_irrreducible_representation,res:infinite_dimensional} shows that we can describe the set of primitive ideals of $\ellone$, even though we do not   know all infinite dimensional algebraically irreducible representations. This set is $\set{\Ker(\pi_{x,\lambda}) : x\in\perpoints,\,\lambda\in\TT}\cup\set{\Ker(\pi_x^1) : x\in\aperpoints}$. As a word of warning, let us note that there are multiple occurrences in this enumeration. If $x_1,x_2\in\perpoints$ and $\lambda_1,\lambda_2\in\TT$, then $\Ker(\pi_{x_1,\lambda_1})=\Ker(\pi_{x_2,\lambda_2})$ if and only if the orbits of $x_1$ and $x_2$ coincide and $\lambda_1=\lambda_2$. If $x_1,x_2\in\aperpoints$, then $\Ker(\pi_{x_1}^1)=\Ker(\pi_{x_2}^1)$ if and only if the closures of the orbits of $x_1$ and $x_2$ coincide. Furthermore, $\set{\Ker(\pi_{x,\lambda}) : x\in\perpoints,\,\lambda\in\TT}\cap\set{\Ker(\pi_x^1) : x\in\aperpoints}=\emptyset$.  We refer to \cite[Lemma~4.14]{de_jeu_tomiyama:2016} for more details.
\end{remark}

We collect the material on arbitrary primitive ideals in the next result. The crucial link with involutive representations is established in the parts~\partref{kernel_of_involutive_representation} and~\partref{kernel_of_topologically_irreducible_involutive_representation}.

\begin{theorem}\label{res:primitive_and_involutive}
Let~$I$ be a primitive ideal of $\ellone$.
\begin{enumerate}
\item\label{intersection}
\begin{enumerate}
\item\label{finite_dimensional}If~$I$ is the kernel of a finite dimensional algebraically irreducible representation of $\ellone$, then $I=\Ker(\pi_{x,\lambda})$ for some $x\in\perpoints$ and $\lambda\in\TT$.

Furthermore, $\pi_{x,\lambda}$ is a topologically irreducible involutive representation of $\ellone$ on a finite dimensional Hilbert space.
\item\label{infinite_dimensional} If~$I$ is the kernel of an infinite dimensional algebraically irreducible representation of $\ellone$, then there exists an infinite invariant subset~$S$ of~$X$ such that
\[
I=\bigcap_{\genfrac{}{}{0pt}{1}{x\in\perpoints\cap S}{\lambda\in\TT}}\!\!\!\!\Ker(\pi_{x,\lambda})\;\;\cap\!\!\!\bigcap_{\genfrac{}{}{0pt}{1}{x\in\aperpoints\cap S}{}}\!\!\!\!\Ker(\pi_x^{2}).
\]
If $X$ is metrizable, then there exists $x\in\aperpoints$ such that $I=\Ker(\pi_{x}^2)$.

Furthermore, all $\pi_{x,\lambda}$ that occur are topologically irreducible involutive representations of $\ellone$ on finite dimensional Hilbert spaces, and all $\pi_{x}^2$ that occur are topologically irreducible involutive representations of $\ellone$ on the Hilbert space $\ell^2(\ZZ)$.
\end{enumerate}
\item\label{kernel_of_involutive_representation} There exists a unital involutive representation~$\pi$ of $\ellone$ on a Hilbert space such that $I=\Ker(\pi)$.
\item\label{kernel_of_topologically_irreducible_involutive_representation} If~$X$ is metrizable, then there exists a topologically irreducible involutive representation~$\pi$ of $\ellone$ on a Hilbert space such that $I=\Ker(\pi)$.
\end{enumerate}
\end{theorem}

\begin{proof}
Part~\partref{finite_dimensional} is contained in \cref{res:involutive_periodic_point}.

The first statement in part~\partref{infinite_dimensional} follows by combining equations~\eqref{eq:kernel_of_infinite_dimensional_algebraically_irreducible_representation}, \eqref{eq:aperiodic_point_representation}, and \eqref{eq:aperiodic_point_kernel_description} for $p=2$. The second statement in part~\partref{infinite_dimensional} follows from the metrizable case in \cref{res:kernel_of_infinite_dimensional_algebraically_irrreducible_representation} and \cref{eq:aperiodic_point_kernel_description} for $p=2$.

Since part~\partref{intersection} shows that, in both cases, $I$ is the simultaneous kernel of a suitable collection of involutive representations of $\ellone$, it is also the kernel of the Hilbert direct sum of the representations in this collection; note that this sum can be defined, since the pertinent representations are all contractive. This establishes part~\partref{kernel_of_involutive_representation}.

Part~\partref{kernel_of_topologically_irreducible_involutive_representation} is immediate from part~\partref{intersection}.

\end{proof}

With part~\partref{kernel_of_involutive_representation} of \cref{res:primitive_and_involutive} available, we can now use the observation that was already mentioned in the introduction: kernels of involutive representations of $\ellone$ are well-behaved when their relation with $\cstar$ is concerned.

\begin{theorem}\label{res:proper_closure}
Let~$I$ be a not necessarily closed proper ideal of $\ellone$. Then the closure $\cstarclosure{I}$ of~$I$ in $\cstar$ is a proper closed ideal of $\cstar$.
\end{theorem}

\begin{proof}
It is elementary that $\cstarclosure{I}$\,\, is an ideal of $\cstar$; it remains to be shown that it is a proper subset of $\cstar$. As explained earlier, we may assume that~$I$ is a primitive ideal of $\ellone$. In that case, \cref{res:primitive_and_involutive} shows that $I=\Ker(\pi)$ for some non-zero involutive representation of $\ellone$ on a Hilbert space. If we let $\pi^{\mathrm e}$ denote the extension of~$\pi$ to a non-zero involutive representation of $\cstar$ on that Hilbert space, then $\Ker(\pi^{\mathrm e})$ is a proper closed subset of $\cstar$. Since $\cstarclosure{I}=\cstarclosure{\Ker(\pi)}\subset\cstarclosure{\Ker(\pi^{\mathrm e})}=\Ker(\pi^{\mathrm e})$, we see that $\cstarclosure{I}$ is also a proper subset of $\cstar$.
\end{proof}

\begin{remark}\quad
\begin{enumerate}
\item
An inspection of the structure of the proof of \cref{res:proper_closure} shows that \cite[Proposition~2.6]{de_jeu_tomiyama:2016} is used, which, in turn, is based on the so-called intersection property of the commutant of $\cont(X)$ in $\ellone$; see \cite[Theorem~3.7]{de_jeu_svensson_tomiyama:2012}. The latter result is the rather non-trivial key result in \cite{de_jeu_svensson_tomiyama:2012}. Thus, in spite of the simplicity of its formulation, \cref{res:proper_closure} seems to be a reasonably deep fact about the relation between $\ellone$ and $\cstar$.
\item
For a not necessarily closed proper ideal~$I$ of $\ellone$, a necessary and sufficient condition for $\cstarclosure{I}$\,\,to be a proper ideal of $\cstar$ is given in \cite[Proposition~4.12]{de_jeu_tomiyama:2013b}. In effect, we have shown that this condition is always satisfied, with a proof that is in the spirit of the proof of \cite[Proposition~4.12]{de_jeu_tomiyama:2013b}.
\end{enumerate}
\end{remark}

As an illustration how \cref{res:proper_closure} can be used to deduce results about $\ellone$ from results about $\cstar$ without resorting to the properties of the dynamical system, we re-establish the following result. It can already be found as \cite[Theorem~4.2]{de_jeu_svensson_tomiyama:2012}.

\begin{corollary}\label{res:simplicity}
The following are equivalent:
\begin{enumerate}
\item The only closed ideals of $\ellone$ are $\set{0}$ and $\ellone$;\label{simple}
\item The only closed selfadjoint ideals of $\ellone$ are $\set{0}$ and $\ellone$;\label{star_simple}
\item~$X$ has an infinite number of points, and the only closed invariant subsets of~$X$ are $\emptyset $ and~$X$.\label{minimal}
\end{enumerate}
\end{corollary}

\begin{proof}
It is clear that~\partref{simple} implies~\partref{star_simple}. If~\partref{minimal} does not hold, then it is not too difficult to construct a non-trivial closed selfadjoint ideal of $\ellone$; we refer to \cite[proof of Theorem~4.2]{de_jeu_svensson_tomiyama:2012} for details. Thus~\partref{star_simple} implies~\partref{minimal}. The hard part is to show that~\partref{minimal} implies~\partref{simple}, and it is here that \cref{res:proper_closure} can be put to good use to be able to apply a result about $\cstar$. Indeed, \cite[Theorem~5.4]{tomiyama_THE_INTERPLAY_BETWEEN_TOPOLOGICAL_DYNAMICS_AND_THEORY_OF_C-STAR-ALGEBRAS:1992} shows that, if~\partref{minimal} holds, then the algebra $\cstar$ has only trivial closed ideals. \cref{res:proper_closure} shows that the same is then true for $\ellone$, which is~\partref{simple}.
\end{proof}

We shall now consider the relation between the closure operation in $\ellone$ and in $\cstar$. We start with the following observation, the first part of which was already alluded to in Section~\ref{sec:introduction}.

\begin{proposition}\label{res:retrieve_from_closure}\quad
\begin{enumerate}
\item\label{kernel_is_ok} If $I$ is the kernel of an involutive representation of $\ellone$ on a Hilbert space, then $I=\cstarclosure{I}\cap\ellone$.
\item\label{inheritance} If $\set{I_\alpha : \alpha
\in A}$ is collection of closed ideals of $\ellone$ such that $I_\alpha=\cstarclosure{I_\alpha}\cap\ellone$ for all $\alpha\in A$, then
\[
\bigcap_{\alpha\in A}I_\alpha=\cstarclosure{\bigcap_{\alpha\in A}I_\alpha}\cap\ellone.
\]
\end{enumerate}
\end{proposition}

\begin{proof}
Suppose that $I=\Ker(\pi)$ for an involutive representation $\pi$ of $\ellone$. As in the proof of \cref{res:proper_closure}, we let $\pi^{\mathrm e}$ denote the extension of $\pi$ to an involutive representation of $\cstar$. Then, again as in that proof, we have
$\cstarclosure{I}=\cstarclosure{\Ker(\pi)}\subset\cstarclosure{\Ker(\pi^{\mathrm e})}=\Ker(\pi^{\mathrm e})$. Hence $\cstarclosure{I}\cap\ellone\subset \Ker(\pi^{\mathrm e})\cap\ellone=\Ker(\pi)=I$. Since obviously $I\subset\cstarclosure{I}\cap\ellone$, part~\partref{kernel_is_ok} has been established.

We turn to part~\partref{inheritance}. Using the properties of the $I_\alpha$ in the final step, we see that
\[
\cstarclosure{\bigcap_{\alpha\in A}I_\alpha}\cap\ellone\subset\lrpars{\bigcap_{\alpha\in A}\cstarclosure{I_\alpha}} \cap\ellone=\bigcap_{\alpha\in A}\lrpars{\cstarclosure{I_\alpha} \cap\ellone}=\bigcap_{\alpha}I_\alpha.
\]
Since the reverse inclusion is obvious, the proof is complete.
\end{proof}

Combining \cref{res:retrieve_from_closure} and part~\partref{kernel_of_involutive_representation} of \cref{res:primitive_and_involutive}, we have the following.

\begin{corollary}\label{res:intersection_of_primitive_ok}
If $I$ is an intersection of primitive ideals of $\ellone$, then $I=\cstarclosure{I}\cap\ellone$.
\end{corollary}

We can now determine when all closed ideals of $\ellone$ can be retrieved from their closure in $\cstar$ as above.

\begin{theorem}\label{res:intersection_closure}
The following are equivalent:
\begin{enumerate}
\item\label{intersection_general} $\elloneclosure{I}=\cstarclosure{I}\cap\ellone$ for every not necessarily closed ideal of $\ellone$;
\item\label{intersection_closed} $I=\cstarclosure{I}\cap\ellone$ for every closed ideal of $\ellone$;
\item\label{all_intersection_of_primitive} Every closed ideal of $\ellone$ is an intersection of primitive ideals of $\ellone$;
\item\label{all_self_adjoint} Every closed ideal of $\ellone$ is a selfadjoint ideal of $\ellone$;
\item\label{all_kernels_of_involutive_representations} Every closed ideal of $\ellone$ is the kernel of an involutive representation of $\ellone$ on a Hilbert space;
\item\label{free} There are no periodic points of~$\sigma$ in~$X$.
\end{enumerate}
\end{theorem}

\begin{proof}
It is clear that~\partref{intersection_general} implies~\partref{intersection_closed}. 

Using the continuity of the inclusion of $\ellone$ in $\cstar$, an application of~\partref{intersection_closed} to $\elloneclosure{I}$ shows that~\partref{intersection_closed} implies~\partref{intersection_general}. 

Certainly~\partref{intersection_closed} implies~\partref{all_self_adjoint}, because all closed ideals of $\cstar$ are selfadjoint. 

The equivalence of~\partref{all_intersection_of_primitive},~\partref{all_self_adjoint}, and~\partref{free} is the content of \cite[Theorem~4.4]{de_jeu_tomiyama:2016}.

We know from part~\partref{kernel_of_involutive_representation} of \cref{res:primitive_and_involutive}
that every primitive ideal of $\ellone$ is the kernel of an involutive representation of $\ellone$ on a Hilbert space. Therefore, assuming \partref{all_intersection_of_primitive}, we see that~\partref{all_kernels_of_involutive_representations} holds by taking a Hilbert direct sum.

It is evident that~\partref{all_kernels_of_involutive_representations} implies~\partref{all_self_adjoint}. 

The proof will be complete once we show that~\partref{all_intersection_of_primitive} implies~\partref{intersection_closed}, and this is immediate from \cref{res:intersection_of_primitive_ok}.
\end{proof}







\bibliography{general_bibliography}


\end{document}